\documentclass[10pt]{amsart}
\usepackage{listings}

\makeatletter
\def\@settitle{\begin{center}%
  \baselineskip14\p@\relax
    \normalfont\LARGE

  \@title
  \end{center}%
}
\makeatother
\usepackage{epigraph}
\usepackage[dvipsnames]{xcolor}
\usepackage[hyperfootnotes=false]{hyperref}
\usepackage{enumerate}
\usepackage{tikz}
\usepackage{tikz-cd}
\usepackage{hyperref}
\usepackage{amsthm,amsfonts,amsmath,amscd,amssymb}
\usepackage{xcolor,float}
\usepackage[all]{xy}
\usepackage{mathrsfs}
\usepackage{graphics,graphicx}
\usepackage{caption}
\usepackage{comment}
\usepackage{array}
\newcolumntype{P}[1]{>{\centering\arraybackslash}p{#1}}
\newcolumntype{M}[1]{>{\centering\arraybackslash}m{#1}}

\usepackage{epigraph}
\let\oldmarginpar\marginpar
\renewcommand\marginpar[1]{\-\oldmarginpar[\raggedleft\footnotesize #1]%
	{\raggedright\footnotesize #1}}
\theoremstyle{plain}

\newtheorem{thm}{Theorem}[section]
\newtheorem{lemma}[thm]{Lemma}
\newtheorem{conj}[thm]{Conjecture}
\newtheorem*{theorem*}{Theorem}
\newtheorem*{corollary*}{Corollary}

\newtheorem{prop}[thm]{Proposition}

\theoremstyle{definition}

\newtheorem{definition}[thm]{Definition}

\newtheorem{remark}[thm]{Remark}

\numberwithin{equation}{section}

\renewcommand{\S}{\mathbb{S}}

\newcommand{\N}{\mathbb{N}}
\newcommand{\Z}{\mathbb{Z}}

\newcommand{\SM}{\mathfrak{S}}

\renewcommand{\c}{c^-}

\renewcommand{\d}{\delta}

\newcommand{\sse}{\subset}
\newcommand{\lr}{\longrightarrow}

\newcommand{\tr}{\operatorname{tr}}

\newcommand{\Arf}{\operatorname{Arf}}

\newcounter{daggerfootnote}

\newcommand{\bC}{\mathbb{C}}

\newcommand{\bQ}{\mathbb{Q}}

\newcommand{\bR}{\mathbb{R}}

\renewcommand{\c}{\mathfrak{c}}

\usepackage{dsfont}
\allowdisplaybreaks

\def \vertbar [#1](#2,#3,#4){
    \draw [#1] (#2,#3) -- (#2,#4);
    \draw [fill=white] (#2,#3) circle [radius=0.1];
    \draw [fill=black] (#2,#4) circle [radius=0.1];
}

\usepackage{mathdots}
\usepackage{afterpage}
\makeatletter
\providecommand{\leftsquigarrow}{%
  \mathrel{\mathpalette\reflect@squig\relax}%
}
\newcommand{\reflect@squig}[2]{%
  \reflectbox{$\m@th#1\rightsquigarrow$}%
}
\makeatother

\makeatletter
\def\Ddots{\mathinner{\mkern1mu\raise\p@
\vbox{\kern7\p@\hbox{.}}\mkern2mu
\raise4\p@\hbox{.}\mkern2mu\raise7\p@\hbox{.}\mkern1mu}}
\makeatother

\def \horline [#1](#2,#3,#4){
    \draw [#1] (#2,#4) -- (#3,#4);
    \draw [fill=white] (#2,#4) circle [radius=0.1];
    \draw [fill=black] (#3,#4) circle [radius=0.1];
}

\def \crossing (#1,#2)(#3,#4){
\draw (#1,#2) -- (#3,#4);
\draw (#1,#4) -- (#3,#2);
}

\DeclareFontFamily{U}{mathb}{}
\DeclareFontShape{U}{mathb}{m}{n}{
  <-5.5> mathb5
  <5.5-6.5> mathb6
  <6.5-7.5> mathb7
  <7.5-8.5> mathb8
  <8.5-9.5> mathb9
  <9.5-11.5> mathb10
  <11.5-> mathbb12
}{}

\usetikzlibrary{decorations.markings, arrows}
\tikzset{tangent/.style={decoration={markings,mark=at position #1 with {
      \coordinate (tangent point-\pgfkeysvalueof{/pgf/decoration/mark info/sequence number}) at (0pt,0pt);
      \coordinate (tangent unit vector-\pgfkeysvalueof{/pgf/decoration/mark info/sequence number}) at (1,0pt);
      \coordinate (tangent orthogonal unit vector-\pgfkeysvalueof{/pgf/decoration/mark info/sequence number}) at (0pt,1);
      }},postaction=decorate},
    use tangent/.style={
        shift=(tangent point-#1),
        x=(tangent unit vector-#1),
        y=(tangent orthogonal unit vector-#1)
    },
    use tangent/.default=1
    }

\definecolor{codegreen}{rgb}{0,0.6,0}
\definecolor{codegray}{rgb}{0.5,0.5,0.5}
\definecolor{codepurple}{rgb}{0.58,0,0.82}
\definecolor{backcolour}{rgb}{0.95,0.95,0.92}

\lstdefinestyle{mystyle}{
    backgroundcolor=\color{backcolour},   
    commentstyle=\color{codegreen},
    keywordstyle=\color{magenta},
    numberstyle=\tiny\color{codegray},
    stringstyle=\color{codepurple},
    basicstyle=\ttfamily\footnotesize,
    breakatwhitespace=false,         
    breaklines=true,                 
    captionpos=b,                    
    keepspaces=true,                 
    numbers=left,                    
    numbersep=5pt,                  
    showspaces=false,                
    showstringspaces=false,
    showtabs=false,                  
    tabsize=2
}

\begin{document}

	\title{A binary invariant of matrix mutation}
	
	\subjclass[2010]{Primary: 13F60. Secondary: 15A15.}

	\author{Roger Casals}
	\address{University of California Davis, Dept. of Mathematics, USA}
	\email{casals@math.ucdavis.edu}

\maketitle
\vspace{-1cm}
\begin{abstract}
We construct a binary mutation invariant for skew-symmetric integer matrices. The invariant is not an integer congruence invariant for matrices of odd size: we provide examples of congruent such matrices with different values for the invariant.
\end{abstract}



\section{Introduction}

The object of this note will be to construct a binary mutation invariant for skew-symmetric integer matrices. The invariant is not a congruence invariant over $\Z$ for integer matrices of odd size. Namely, for any odd number, there are skew-symmetric integer matrices of that size which are congruent over $\Z$ and yet their binary invariants differ.\\

{\bf Scientific context}. Let $M_{n}(\Z)$ denote the set of $n\times n$ integer matrices. Let $B\in M_{n}(\Z)$ be an $n\times n$ skew-symmetric integer matrix. The first definition of matrix mutation in the mathematical literature is \cite[Section 4]{FominZelevinsky_ClusterI} and we follow the presentation\footnote{Matrix mutation is defined more generally for skew-symmetrizable matrices, cf.~\cite[Section 2.7]{FWZ}. From loc.~cit., rank is preserved under mutation by \cite[Lemma 3.2]{BFZ05}. In this note we focus on (square) skew-symmetric matrices.} in \cite[Section 2.7]{FWZ}. Let $k\in[n]$ be the column index at which we mutate. Consider the following two matrices $J_{k,n},E_{k,n}\in M_n(\Z)$:
\begin{itemize}
\item[-] $J_{k,n}$ denotes the diagonal matrix of size $n\times n$ whose diagonal entries are all $1$, except for the $(k, k)$ entry, which is $-1$.

\item[-] $E_{k,n} = (e_{ij} )$ is the $n\times n$ matrix with $e_{ik} := \mbox{max}(0, -b_{ik})$ for all $i\in[n]$ and all
other entries equal to 0.
\end{itemize}
Let us denote $M_{k}:=J_{k,n}+E_{k,n}$, which depends on $k\in[n]$ and $B$. We refer to $M_{k}$ as the replicating matrix of $B$ at $k\in[n]$. By definition, the mutation of $B$ at $k$ is the matrix
$$\mu_k(B):=M_{k}\cdot B\cdot M_{k}^t.$$
The mutated matrix $\mu_k(B)$ is skew-symmetric and thus mutation can be iterated. Note that mutation at $k$ is involutive, i.e. $\mu_k(\mu_k(B))=B$ for any $k\in[n]$. Typically $\mu_l(\mu_k(B))\neq B$ if $l\neq k$, $l,k\in[n]$. Two matrices $B_1,B_2\in M_n(\Z)$ are said to be equal up to a simultaneous permutation of rows and columns if there exists a permutation matrix $P\in M_n(\Z)$ such that $B_2=P\cdot B_1\cdot P^t$. We use the following definition from \cite[Definition 2.8.1]{FWZ}:

\begin{definition}\label{def:mutation}
Two skew-symmetric matrices $B,B'\in M_n(\Z)$ are mutation equivalent if one can get from $B$ to $B'$ by a sequence of mutations,
possibly followed by a simultaneous permutation of rows and columns.\hfill$\Box$
\end{definition}
Definition \ref{def:mutation} must be put in contrast with the following notion:

\begin{definition}\label{def:congruent}
Two integer matrices $B,B'\in M_n(\Z)$ are said to be congruent over $\Z$ if there exists a unimodular matrix $X\in M_n(\Z)$ such that $B'=X\cdot B\cdot X^t$.\hfill$\Box$
\end{definition}

Throughout this note, congruent always means congruent using only unimodular matrices $X\in M_n(\Z)$ with integer entries, as in Definition \ref{def:congruent} above: e.g.~non-integer entries in $\bQ,\bR$ or $\bC$ are not allowed for such $X$. It follows from Definition \ref{def:mutation} that two skew-symmetric matrices $B,B'\in M_n(\Z)$ that are mutation equivalent must be congruent in $M_n(\Z)$, over $\Z$, as in Definition \ref{def:congruent}. Therefore, any invariant of congruence (over $\Z$) is an invariant of mutation. The invariant factors of the Smith normal form (thus the determinant), and the rank are congruence invariants, see \cite{Newman97}.\footnote{The greatest common divisor of the matrix entries in each column is also a mutation invariant.} The present note is an attempt to use linear algebra to construct an invariant of matrix mutation that is {\it not} an invariant of matrix congruence over $\Z$.

For $n = 2$, mutation negates the entries of the matrix $B$ for any $k\in\{1,2\}$. It is thus immediate to decide whether $B,B'\in M_2(\Z)$ are mutation equivalent. The article \cite{ABB08} presents an algorithm for determining whether two $3\times 3$ skew-symmetric matrices are mutation equivalent, addressing the $n=3$ case. As per usual, the finite type case is better understood, cf.~\cite[Chapter 5]{FWZ}, e.g.~\cite[Theorem 5.11.3.(4)]{FWZ}: the goal is to start exploring the general case, beyond the finite or affine cases. See also the recent work \cite{FominNeville23} on long mutation cycles. Beyond what is listed above, and properties such as the existence of a reddening sequence or a rigid potential \cite[Corollary 8.10]{DWZ}, the author knows of no general mutation invariant, which is not a congruence invariant (over $\Z$), for higher $n\geq4$. We believe this to be an important problem relevant to the theory of cluster algebras, cf.~\cite[Problem 2.8.2]{FWZ} and \cite[Slide 9]{Fomin22}.\\

{\bf The $\delta$-invariant and main result}. Let $B\in M_{n}(\Z)$ be an $n\times n$ skew-symmetric integer matrix with entries $B=(b_{ij})$, $i,j\in[n]$. Consider the unipotent matrix $V(B)\in M_{n}(\Z)$ defined by
$$V(B)_{i,j} := \begin{cases}
  b_{ij}  & \text{ if }i<j \\
  1  & \text{ if }i=j \\
  0  & \text{ if }i>j \\
\end{cases}
$$
This matrix satisfies $B=V(B)-V(B)^t$ because $B$ is skew-symmetric. Similarly, consider the symmetric matrix $\mathfrak{S}(B)\in M_{n}(\Z)$ defined by $\mathfrak{S}(B):=V(B)+V(B)^t$. In the notation of \cite[Definition 5.11.2]{FWZ}, $\mathfrak{S}(B)$ is a quasi-Cartan companion of $B$.

\begin{definition}
Let $B\in M_{n}(\Z)$ be an $n\times n$ skew-symmetric integer matrix. By definition, the $\delta$-invariant $\delta(B)$ of $B$ is
$$\delta(B):=\det(\mathfrak{S}(B))\pmod{4}.$$
That is, $\delta(B)$ is the determinant of $\mathfrak{S}(B)$ modulo 4.\hfill$\Box$
\end{definition}

\begin{remark}
For $n\in\Z$ odd, the $\delta$-invariant is binary valued: $\delta(B)\equiv0,2\pmod{4}$ if $n$ is odd. Indeed, $\mathfrak{S}(B)\equiv B\pmod{2}$ and $\det(B)\equiv0\pmod{2}$ because $\det(B)=0$, over $\Z$, as $B$ is skew-symmetric. Therefore $\delta(B)$ is always even in this case. For $n\in\Z$ even, computations indicate that the $\delta$-invariant might still be binary: for $n\equiv 0\pmod{4}$, then $\delta(B)\equiv0,1\pmod{4}$ and for $n\equiv 2\pmod{4}$, then $\delta(B)\equiv 0,3\pmod{4}$.\hfill$\Box$
\end{remark}

The following property, showing that the $\delta$-invariant remains unchanged under simultaneous permutation of rows and columns, will be proven in Section \ref{ssec:perminv}.

\begin{lemma}\label{lem:perm_invariance}
Let $B\in M_{n}(\Z)$ be an $n\times n$ skew-symmetric integer matrix. Then, for any permutation matrix $P\in M_n(\Z)$, the equality $\delta(PBP^t)\equiv\delta(B)\pmod{4}$ holds.\hfill$\Box$
\end{lemma}

The main result of this article reads as follows:

\begin{thm}\label{thm:main}
Let $B\in M_{n}(\Z)$ be an $n\times n$ skew-symmetric integer matrix. If $B'\in M_{n}(\Z)$ is mutation equivalent to $B$, then $$\delta(B')\equiv\delta(B)\pmod{4}.$$
\hfill$\Box$
\end{thm}

Theorem \ref{thm:main} is proven in Section \ref{ssec:mutinv}. In contrast to Theorem \ref{thm:main}, we show that this $\delta$-invariant is {\it not} a congruence invariant over $\Z$ if $n\in\Z$ is odd:

\begin{prop}\label{prop:not_cong}
Let $n\in\Z$ be odd, $n\geq3$. Then, there exist two skew-symmetric integer matrices $B,B^\circ\in M_n(\Z)$ such that:
\begin{itemize}
    \item[(i)] $B$ and $B^\circ$ are congruent over $\Z$, i.e.~ $\exists X\in M_n(\Z)$ unimodular such that $B^\circ=X\cdot B\cdot X^t$,

    \item[(ii)] $\delta(B)\not\equiv\delta(B^\circ)\pmod{4}$.
\end{itemize} 
\end{prop}

 Proposition \ref{prop:not_cong} is proven in Section \ref{ssec:prop_not_cong}. A computer randomly generating two congruent skew-symmetric matrices of odd size quickly finds examples of such $B,B^\circ$ with different $\delta$-invariants.\footnote{The author has implemented such computer calculations in SageMath, see Appendix (Section \ref{sec:app}).} I do not know of two congruent skew-symmetric integer matrices of even size with different $\delta$-invariants.\\

{\bf Acknowledgements}. We thank M.~Sherman-Bennett and D.~Weng for helpful comments. We are grateful to S.~Fomin for useful feedback on an initial draft. R.~Casals is supported by the NSF CAREER DMS-1942363, a Sloan Research Fellowship of the Alfred P Sloan Foundation and a UC Davis College of L\&S Dean's Fellowship.\hfill$\Box$

{\bf Notation}. We denote $[n]:=\{1,2,\ldots,n\}\sse\N$ and write $M_n(\Z)$ for the set of $n\times n$ integer matrices. The determinant and trace of a matrix $A$ are respectively denoted $\det(X)$ and $\tr(X)$. The transpose of a matrix $X$ is denoted by $X^t$ and its adjugate by $\mbox{adj}(X)$.\hfill$\Box$

\section{Invariance properties of $\delta(B)$}\label{sec:invariance}
In this section we prove Lemma \ref{lem:perm_invariance} and Theorem \ref{thm:main}, concluding that the $\delta$-invariant is an invariant of matrix mutation. Invariance under simultaneous permutations of rows and columns is established in Subsection \ref{ssec:perminv}. Mutation invariance is proven in Subsection \ref{ssec:mutinv}.
\subsection{Three lemmas about matrices} The arguments we present to prove Lemma \ref{lem:perm_invariance} and Theorem \ref{thm:main} use the following three lemmas. Lemmas \ref{lem:1} and \ref{lem:2} are general linear algebra, not to do with matrix mutation. Lemma \ref{lem:3} uses features specific to matrix mutation and replicating matrices.
\begin{lemma}\label{lem:1}
Let $Z\in M_n(\Z)$ be a skew-symmetric. Suppose that $V,W\in M_n(\Z)$ satisfy $Z=V-V^t$ and $Z=W-W^t$. The $V-W$ is symmetric.
\end{lemma}

\begin{proof}
This follows from $(V-W)^t=V^t-W^t=(V-Z)+(Z-W)=V-W$.
\end{proof}

\begin{lemma}\label{lem:3}
Let $B\in M_{n}(\Z)$ be an $n\times n$ skew-symmetric integer matrix and $M_k$ its matrix replicating matrix at $k\in[n]$. Then the $i$th diagonal term $\tau_{ii}$, $i\in[n]$, of the difference
$$M_k\cdot\SM(B)\cdot M_k^t-\SM(\mu_k(B))$$
is of the form $\tau_{ii}=2t_{ii}$ where
$$t_{ii}:= \begin{cases}
  (b_{ik}+\max(0,b_{ik}))\max(0,b_{ik})  & \text{ if }i<k, \\
  \qquad \qquad \qquad 0  & \text{ if }i=k, \\
    (b_{ik}+\max(0,-b_{ik}))\max(0,-b_{ik})   & \text{ if }i>k.
\end{cases}
$$
In particular, $t_{ii}\equiv0\pmod{2}$ and $\tau_{ii}\equiv0\pmod{4}$ for all $i\in[n]$.
\end{lemma}

\begin{proof}
Consider the unipotent matrix $V(B)$ and $M_k$: we first compute the diagonal entries of $M_k\cdot V(B)\cdot M_k^t$. The diagonal entries of $A:=V(B)\cdot M_k^t$ are as follows:
$$A_{ii}= \begin{cases}
  b_{ik}\max(0,b_{ik})+1  & \text{ if }i<k, \\
  \qquad \quad -1  & \text{ if }i=k, \\
\qquad \quad    1  & \text{ if }i>k,
\end{cases}
$$
for $i\in[n]$. The $k$th row of $A$ is the transpose of the $k$th column of $M_k$, i.e.~$A_{ki}=\mbox{max}(0, -b_{ik})$ for $i\in[n]$. Since $M_k$ coincides with the identity away from the $k$th column, the diagonal entries of $A$ and its $k$th row suffice to compute the diagonal entries of the product $M_kA=M_kV(B)M_k^t$. These diagonal entries of $M_kV(B)M_k^t$ are:
$$(M_kV(B) M_k^t)_{ii}= \begin{cases}
  (b_{ik}+\max(0,b_{ik}))\max(0,b_{ik})+1  & \text{ if }i<k, \\
  \qquad \qquad \qquad 1  & \text{ if }i=k, \\
    (b_{ik}+\max(0,-b_{ik}))\max(0,-b_{ik})+1   & \text{ if }i>k.
\end{cases}
$$
The diagonal entries of $V(\mu_k(B))$ are all 1, as this is a unipotent matrix. Therefore, the diagonal entries of the difference $M_kV(B)M_k^t-V(\mu_k(B))$ are:
$$(M_kV(B)M_k^t-V(\mu_k(B)))_{ii}= \begin{cases}
  (b_{ik}+\max(0,b_{ik}))\max(0,b_{ik})  & \text{ if }i<k, \\
  \qquad \qquad \qquad 0  & \text{ if }i=k, \\
    (b_{ik}+\max(0,-b_{ik}))\max(0,-b_{ik})   & \text{ if }i>k.
\end{cases}
$$
Since $\SM(B)=V(B)+V(B)^t$ and $\SM(\mu_k(B))=V(\mu_k(B))+V(\mu_k(B))^t$, the diagonal terms $\tau_{ii}$ of the difference $M_k\cdot\SM(B)\cdot M_k^t-\SM(\mu_k(B))$ are twice the diagonal terms of the difference $M_kV(B)M_k^t-V(\mu_k(B))$. This proves the required identity in the statement for the diagonal terms $\tau_{ii}$.\\

\noindent In order to conclude that $t_{ii}$ is even, and thus $\tau_{ii}\equiv0\pmod{4}$ for all $i\in[n]$, it suffices to note that
$$(b_{ik}+\max(0,b_{ik}))\max(0,b_{ik})\equiv
\begin{cases}
  (b_{ik}+0)\cdot 0\equiv0  & \text{ if }\max(0,b_{ik})=0, \\
  \qquad \quad 2b_{ik}^2  & \text{ if }\max(0,b_{ik})\neq0.
\end{cases}
$$
Therefore $t_{ii}$ is always even in either case and $\tau_{ii}\equiv0\pmod{4}$ for all $i\in[n]$. The argument for $i>k$, i.e.~showing $(b_{ik}+\max(0,-b_{ik}))\max(0,-b_{ik})\equiv 0\pmod{2}$ and thus $\tau_{ii}\equiv0\pmod{4}$, is identical.
\end{proof}

\begin{lemma}\label{lem:2}
Let $Y\in M_n(\Z)$ a symmetric matrix. Suppose that $Y$ has all diagonal terms equal to zero, i.e.~ $Y=(y_{ij})$ satisfies $y_{ii}=0$ for all $i\in[n]$. Then:
\begin{enumerate}
    \item $\tr(XY)\equiv 0\pmod{2}$ for any symmetric matrix $X\in M_n(\Z)$, i.e.~$\tr(XY)$ is even if $X=X^t$.
    \item $\det(R)\equiv\det(R+2Y)\pmod{4}$ for any symmetric matrix $R\in M_n(\Z)$.
\end{enumerate}

\end{lemma}
\begin{proof} For Part (1), consider the unipotent matrix $V:=V(Y)$. Then $Y\equiv V-V^t\pmod{2}$. Therefore $$XY\equiv X(V-V^t)\equiv XV-XV^t\pmod{2}.$$
Since $X=X^t$ is symmetric, $XV^t=X^tV^t=(VX)^t$. Hence
$$\tr(XY)\equiv \tr(XV-(VX)^t)\equiv \tr(XV)-\tr((VX)^t)\equiv \tr(XV)-\tr(VX)\equiv 0\pmod{2},$$
where we have used $\tr(A^t)=\tr(A)$ and the cyclicity $\tr(AF)=\tr(FA)$ for any $A,F\in M_n(\Z)$.

For Part (2), consider the expansion of the determinant
\begin{equation}\label{eq:taylordet}
\det(R+tY)=\det(R)+\tr(\mbox{adj}(R)Y)\cdot t+\mathcal{O}(t^2),
\end{equation}
where $\mbox{adj}(R)$ denotes the adjugate of $R$. Since $R$ is symmetric, its adjugate $X:=\mbox{adj}(R)$ is symmetric. By Part (1), $\tr(\mbox{adj}(R)Y)$ is even. The identity above thus yields
$$\det(R+2Y)=\det(R)+2\tr(\mbox{adj}(R)Y)+\mathcal{O}(4),$$
where $2\tr(\mbox{adj}(R)Y)$ is divisible by $4$ and $\mathcal{O}(4)$ indicates all other terms, also divisible by $4$. Note that the expansion (\ref{eq:taylordet}) of the determinant is valid in finite characteristic if we use Hasse derivatives to compute the coefficients. In conclusion, reduction modulo 4 yields the required equality
$$\det(R+2Y)\equiv\det(R)\pmod{4}.$$
\end{proof}

\begin{remark}
I thank L.~Shen, who pointed out to me that Lemma \ref{lem:2}.(2) can be proven alternatively by showing that
$\det(R)\equiv\det(R+2(\varepsilon_{ij}+\varepsilon_{ji}))\pmod{4}$ via direct computation, where $R\in M_n(\Z)$ is any symmetric matrix and $\varepsilon_{ij}\in M_n(\Z)$ is the $(i,j)$-elementary matrix, with zero entries everywhere except for the $(i,j)$-entry, which equals 1.\hfill$\Box$
\end{remark}
\subsection{Proof of Lemma \ref{lem:perm_invariance}}\label{ssec:perminv} We now show that the $\delta$-invariant remains unchanged under simultaneous permutation of rows and columns. Let us denote by $\varepsilon_{ij}\in M_n(\Z)$ the elementary matrix with zero entries everywhere except for the $(i,j)$-entry, which equals 1. Then:

\begin{lemma}\label{lem:perm1}
Let $P\in M_n(\Z)$ be the permutation matrix associated to the simple transposition $s_k$, $k\in[n-1]$. Then $P\SM(B)P^t-\SM(PBP^t)=2b_{k,k+1}(\varepsilon_{k,k+1}+\varepsilon_{k+1,k})$.
\end{lemma}

\begin{proof}
Since $\SM(B)=V(B)+V(B)^t$, it suffices to show the following identity:
\begin{equation}\label{eq:PVP}
\tag{*}
PV(B)P^t-V(PBP^t)=b_{k,k+1}(\varepsilon_{k,k+1}+\varepsilon_{k+1,k}).
\end{equation}
Given any matrix $A\in M_n(\Z)$, all the rows and columns of $PAP^t$ coincide with those of $A$ except for the $k$th and $(k+1)$st rows and $k$th and $(k+1)$st columns. It thus suffices to study the difference of $PV(B)P^t-V(PBP^t)$ at those rows and columns. By construction, the entries of these two matrices $V(PBP^t),PV(B)P^t$ coincide except at the entries $(k+1,k)$ and $(k,k+1)$. These entries are:
$$(PV(B)P^t)_{i,j} := \begin{cases}
  b_{k,k+1}  & \text{ if }(i,j)=(k+1,k) \\
  0  & \text{ if }(i,j)=(k,k+1) \\
\end{cases},\quad (V(PBP^t))_{i,j} := \begin{cases}
  0  & \text{ if }(i,j)=(k+1,k) \\
  -b_{k,k+1}  & \text{ if }(i,j)=(k,k+1). \\
\end{cases}
$$
This implies Identity (\ref{eq:PVP}) above and thus the lemma.
\end{proof}

\noindent Let us now use Lemma \ref{lem:perm1} to show Lemma \ref{lem:perm_invariance}, as follows. Since any permutation can be written as a product of simple transpositions, it suffices to show that $\delta(B)\equiv\delta(PBP^t)\pmod{4}$ for $P$ the permutation matrix of a simple transposition. For that, note that the matrix
$$T:=-b_{k,k+1}(\varepsilon_{k,k+1}+\varepsilon_{k+1,k})$$ has all diagonal terms equal to zero; it is also symmetric. By Lemma \ref{lem:perm1} above, Lemma \ref{lem:2}.(2) can be applied to the symmetric matrix $R:=P\SM(B)P^t$ and the matrix $Y:=T$. Lemma \ref{lem:2}.(2) implies:
$$\det(\SM(PBP^t))\stackrel{\ref{lem:perm1}}{\equiv}\det(P\SM(B)P^t+2\cdot T)\stackrel{\ref{lem:2}.(2)}{\equiv}\det(P\SM(B)P^t)\equiv \det(\SM(B))\pmod{4}.$$
Thus, for any permutation $P\in M_n(\Z)$, the equality $\delta(PBP^t)\equiv\delta(B)\pmod{4}$ is satisfied.\hfill$\Box$


\subsection{Proof of Theorem \ref{thm:main}}\label{ssec:mutinv} The goal is to show $$\det(\SM(\mu_k(B)))\equiv \det(\SM(B))\pmod{4}$$
for any $k\in[n]$. Indeed, Lemma \ref{lem:perm_invariance} has already shown that the $\delta$-invariant remains unchanged under simultaneous permutations of rows and columns. It therefore suffices to establish invariance under matrix mutation.

Note that we have the identity
$\mu_k(B)=V(\mu_k(B))-V(\mu_k(B))^t$
and the identity
$$\mu_k(B)=M_kBM_k^t=M_k(V(B)-V(B)^t)M_k^t=M_kV(B)M_k^t-(M_kV(B)M_k^t)^t.$$
By Lemma \ref{lem:1} applied to $Z:=\mu_k(B)$, $V:=M_kV(B)M_k^t$ and $W:=V(\mu_k(B))$, the difference
$$T:=M_kV(B)M_k^t-V(\mu_k(B))$$
is a symmetric matrix $T\in M_n(\Z)$. Therefore

\begin{equation}\label{eq:1}
 \tag{**}
\begin{split}
M_k\SM(B)M_k^t-\SM(\mu_k(B)) & = M_k(V(B)+V(B)^t)M_k^t - (V(\mu_k(B))+V(\mu_k(B))^t)\\
 & = (M_kV(B)^tM_k^t-V(\mu_k(B))^t)+(M_k(V(B))M_k^t-V(\mu_k(B)))\\
 & = 2T.
\end{split}
\end{equation}
By Lemma \ref{lem:3}, the diagonal entries $t_{ii}$ of $T$ are even for all $i\in[n]$. Therefore the reduction of $T$ modulo 2 is a symmetric matrix with zero diagonal terms and the reduction of $2T$ modulo 4 is a symmetric matrix with zero diagonal terms. Let $T^\circ\in M_n(\Z)$ be the unique integer matrix whose entries are $0$ or $1$ and $T^\circ\equiv T\pmod{2}$. In particular, all diagonal terms of $T^\circ$ are zero. We ease notation with $M:=M_k$. Let us apply Lemma \ref{lem:2}.(2) to $R:=M\SM(B)M^t$ and $Y:=T^\circ$. Then
$$\det(M\SM(B)M^t)\stackrel{\ref{lem:2}.(2)}{\equiv} \det(M\SM(B)M^t+2T^\circ)\pmod{4}.$$
By construction, $2T\equiv 2T^\circ\pmod{4}$ and thus $\det(M\SM(B)M^t)\equiv \det(M\SM(B)M^t+2T)\pmod{4}$. Since $M$ is unimodular, we also have $\det(M\SM(B)M^t)=\det(\SM(B))$. Therefore
$$\det(\SM(\mu_k(B)))\stackrel{(\ref{eq:1})}{\equiv}\det(M\SM(B)M^t+2T)\stackrel{\ref{lem:2}.(2)}{\equiv}\det(M\SM(B)M^t)\equiv \det(\SM(B))\pmod{4}.$$
In conclusion, $\det(\SM(B))$ modulo 4 is a mutation invariant. \hfill$\Box$

\begin{remark}
At the core of the argument is the equality
$$M\SM(B)M^t-\SM(\mu_k(B))= 2T,$$
where $T$ has all diagonal entries even and $M=M_k$ is a replicating matrix for $B$. Let $\S\sse M_n(\Z)$ denote the subset of symmetric matrices. Any function $f:\S\lr \mathcal{C}$ to a set $\mathcal{C}$ such that $$f(S)=f(MSM^t-2T),\quad \forall T\in\S\mbox{ such that }T_{ii}=0\quad \forall i\in[n],$$
gives a candidate mutation invariant $\delta_f(B):=f(\SM(B))$. There are choices of $f$ that lead to trivial invariants. For instance, if $f(A)=\tr(A)\pmod{2}$, then the associated invariant $\delta_f(B)=\tr(\SM(B))\pmod{2}$ would be always zero, since the diagonal terms of $\SM(B)$ are all even. Note that any $f$ which is a congruence invariant and satisfies the equality in Lemma \ref{lem:2}.(2) must satisfy the condition above. The binary $\delta$-invariant we defined above corresponds to $f(A)=\det(A)\pmod{4}$. Indeed, the determinant modulo 4 satisfies the condition above by virtue of being a congruence invariant and Lemma \ref{lem:2}.(2). Finally, we remark that the multiset of invariant factors of the Smith normal form (even modulo 4), which is a congruence invariant, does not satisfy the equality in Lemma \ref{lem:2}.(2), even if one restricts to symmetric matrices with $2$ in the diagonal entries.\hfill$\Box$
\end{remark}


\section{Examples and comments}\label{sec:examples}

The following are some computations and remarks with regards to the $\delta$-invariant. Proposition \ref{prop:not_cong} is proven in Subsection \ref{ssec:prop_not_cong} below.

\subsection{Two $5\times 5$ congruent matrices with different $\d$-invariants} Consider the two matrices

$$B=\left(\begin{array}{rrrrr}
0 & 5 & 26 & 101 & 74 \\
-5 & 0 & 10 & 38 & 27 \\
-26 & -10 & 0 & 27 & 34 \\
-101 & -38 & -27 & 0 & 83 \\
-74 & -27 & -34 & -83 & 0
\end{array}\right),\qquad
B'=\left(\begin{array}{rrrrr}
0 & 693 & 6624 & 4074 & -8446 \\
-693 & 0 & 1136 & 1853 & -4238 \\
-6624 & -1136 & 0 & 11029 & -26677 \\
-4074 & -1853 & -11029 & 0 & -2349 \\
8446 & 4238 & 26677 & 2349 & 0
\end{array}\right).$$

The matrices $B,B'$ are congruent $B'=XBX^t$ via the matrix

$$X=\left(\begin{array}{rrrrr}
0 & 2 & -3 & 4 & -13 \\
0 & 1 & -2 & 4 & -11 \\
-1 & 4 & -11 & 32 & -85 \\
-1 & -2 & 0 & 13 & -31 \\
1 & 2 & 4 & -25 & 56
\end{array}\right),
$$
where $\det(X)=1$. By virtue of being congruent, the Smith elementary divisors of $B$ and $B'$ coincide, even with multiplicity. They are 1 with multiplicity four and 0 with multiplicity one. Similarly, the determinant of both $B,B'$ is 0 and their rank is 4. For both $B,B'$, the greatest common divisors of each of their columns all equal 1. The $\d$-invariants are $\d(B)=0$ and $\d(B')=2$. The appendix contains more examples of such congruent pairs $B,B'$, generated randomly with SageMath, and the code to generate them.

\subsection{A simple example in all odd sizes}\label{ssec:prop_not_cong} Let us prove Proposition \ref{prop:not_cong}. Consider the matrix $A_n\in M_n(\Z)$ for any odd $n\in\N$:
$$A_n=\sum_{i=1}^n\varepsilon_{i,i+1}-\sum_{i=1}^n\varepsilon_{i+1,i},$$
where $\varepsilon_{ij}$ is the $(i,j)$-elementary matrix. It is skew-symmetric. A direct computation shows
$$\delta(A_n)=\begin{cases}
  0  & \text{ if }n\equiv3\pmod{4}, \\
  2  & \text{ if }n\equiv1\pmod{4}.
\end{cases}$$

Consider the unimodular matrix $X_n\in M_n(\Z)$ defined as $X_n:=\mbox{Id}_n+\varepsilon_{1n}$. Then, congruence by $X$ affects the $\delta$-invariant of $A_n$ as follows:
$$\delta(X_nA_nX_n^t)=\begin{cases}
  2  & \text{ if }n\equiv3\pmod{4}, \\
  0  & \text{ if }n\equiv1\pmod{4}.
\end{cases}$$
Indeed, for $n\geq5$, $X_nA_nX_n^t=A_n+(\varepsilon_{n-1,1}-\varepsilon_{1,n-1})$ and the above values of the $\delta$-invariant follow. The case $n=3$ is also covered by this construction. This concludes Proposition \ref{prop:not_cong}: for any odd $n\in\N$, $n\geq3$, the $\delta$-invariant is not a congruence invariant.

\subsection{Comparison to Markov constant in rank 3}\label{sssec:compare_Markov} In the rank three case $n=3$, we can write
$$B:=\left(\begin{array}{rrr}
0 & p & q \\
-p & 0 & r \\
-q & -r & 0
\end{array}\right)$$
for some integers $p,q,r\in\Z$. The articles \cite{ABB08,FT19} introduce and study an integer quantity $C(B)$ to each skew-symmetric matrix $B$ which is invariant under mutation. Consider the Markov constant $C(p,q,r)=p^2+q^2+r^2-pqr$ associated to $p,q,r\in\Z$. The mutation invariant is defined as $C(B):=C(p,q,-r)$ if the underlying quiver $Q(B)$ is acyclic, and as $C(B):=C(p,q,r)$ if the underlying quiver $Q(B)$ contains a cycle. A direct computation shows that
$$\det(\SM(B))=2 \, p q r - 2 \, p^{2} - 2 \, q^{2} - 2 \, r^{2} + 8.$$
Therefore $\d(B)\equiv 2\cdot C(B) \pmod{4}$.

\subsection{Comparison to Arf invariant}
The Arf invariant of a quadratic form in a $\Z_2$-vector space is a binary invariant. It might be natural to compare the $\delta$-invariant to the Arf invariant of a quadratic form associated to an integer skew-symmetric matrix $B\in M_n(\Z)$. Let $N\sse \Z_2^n$ be a free $\Z_2$-submodule of rank $r=\mbox{rk}(B)$ such that $B|_N$ is a non-degenerate skew-symmetric bilinear form. Consider the basis $v_1,\ldots,v_n$ of vectors of $\Z_2^n$ in which the matrix $B$ represents a non-degenerate bilinear skew-symmetric form and assume that $N=\langle v_1,\ldots,v_r\rangle$. Consider the unique quadratic refinement $q_B:N\lr\Z_2$ of $B$ such that
$$q_B(v_i)=1,\qquad q_B(v+w)=q_B(v)+q_B(w)+v^tBw,\quad \forall v,w\in N,$$
where the vectors are understood as column vectors. We define the Arf invariant $\Arf(B)$ of $B$ (restricted to $N$) to be the Arf invariant $\Arf(q_B)$ of this quadratic refinement. See \cite[Chapter IV.4.7]{Knus91} for the definition of the Arf invariant of a quadratic form in characteristic 2. Consider the following two skew-symmetric matrices $B,B'\in M_n(\Z)$:

$$B=\left(\begin{array}{rrrrr}
0 & 0 & 1 & 1 & 0 \\
0 & 0 & -1 & 0 & 0 \\
-1 & 1 & 0 & -1 & 0 \\
-1 & 0 & 1 & 0 & 1 \\
0 & 0 & 0 & -1 & 0
\end{array}\right),\quad B'=\left(\begin{array}{rrrrr}
0 & 1 & 1 & 1 & 0 \\
-1 & 0 & -1 & 0 & 0 \\
-1 & 1 & 0 & -1 & 0 \\
-1 & 0 & 1 & 0 & 1 \\
0 & 0 & 0 & -1 & 0
\end{array}\right).$$

Their ranks are $\mbox{rk}(B)=\mbox{rk}(B)=4$. It is verified that $\delta(B)=2$ and $\delta(B')=0$ by direct computation. Let us now argue that $\Arf(B)=\Arf(B')=1$, and therefore the $\delta$-invariant $\delta(B)$ is different than the above defined $\Arf(B)$. We compute $\Arf(B)$ and $\Arf(B')$ by choosing symplectic bases of $B$ and $B'$.

The following vectors $\{e_1,f_1,e_2,f_2\}$ are a symplectic basis for $B$:
$$e_1=\left(0,\,0,\,1,\,0,\,0\right),\quad e_2=
\left(1,\,1,\,0,\,0,\,0\right),\quad
f_1=\left(0,\,1,\,0,\,0,\,0\right),\quad
f_2=\left(0,\,1,\,0,\,1,\,0\right).$$
That is, they satisfy $e_iBe_j=0$, $f_iBf_j=0$, $e_iBf_j=\delta_{ij}$, for all $i,j\in[2]$. In such a sympectic basis, the Arf invariant can be computed as
$$\Arf(q_B)=q_B(e_1)q_B(f_1)+q_B(e_2)q_B(f_2)=(1\cdot 1)+(1+1+B_{12})\cdot(1+1+B_{24})=1$$

Similarly, the following vectors $\{e'_1,f'_1,e'_2,f'_2\}$ are a symplectic basis for $B$:
$$e'_1=\left(1,\,0,\,0,\,0,\,0\right),\quad
e'_2=\left(0,\,-1,\,0,\,1,\,0\right),\quad 
f'_1=\left(0,\,1,\,0,\,0,\,0\right),\quad
f'_2=\left(0,\,0,\,0,\,0,\,1\right).$$
Then $\Arf(q_{B'})=q_{B'}(e'_1)q_{B'}(f'_1)+q_{B'}(e'_2)q_{B'}(f'_2)=(1\cdot 1)+(1+1+B'_{24})\cdot1=1$.

\subsection{Case of quivers from plabic fences}\label{sssec:compare_Alexander} Let $B=B_\beta\in M_n(\Z)$ be the skew-symmetric matrix associated to a plabic fence with positive braid word $\beta$, cf.~\cite[Section 12]{FPST}.\footnote{Here $B$ is considered as a square matrix, with no frozen variables. For simplicity, plabic fences will be considered with all vertical edges with black on top and white on bottom.} Suppose that the 0-framed closure of $\beta$ is a knot, i.e.~the permutation associated to $\beta$ acts transitively on $[n]$. Then $\d(B)\equiv \Delta_\beta(-1)\pmod{4}$ where $\Delta_\beta$ is the Alexander polynomial of the knot $K_\beta$. Note that the Arf invariant of the knot $K_\beta$ is $\Arf(K_\beta)=0$ if $\Delta_\beta(-1)\equiv \pm1\pmod{8}$ and $\Arf(K_\beta)=1$ if $\Delta_\beta(-1)\equiv \pm3\pmod{8}$. In these cases, the $\delta$-invariant does not necessarily determine the Arf invariant. To wit, present the $(2,3)$ and the $(2,7)$-torus knots as the 0-framed closures of the 2-stranded positive braid words $\beta_1=\sigma_1^3$ and $\beta_2=\sigma_1^7$, respectively. These give examples of $B,B'$ with different Arf invariants but same $\d$-invariant.

\begin{remark}
The existence or non-existence of a reddening sequence is another binary mutation invariant of $B$, see \cite[Corollary 19]{Muller16}. It is a non-trivial invariant. It does not coincide with the $\delta$-invariant. To wit, the $A_n$ linear quiver admits a reddening sequence for all $n\in\N$. As stated in Subsection \ref{ssec:prop_not_cong}, for $n\equiv1\pmod{4}$, its $\delta$-invariant is $2$. For $n\equiv3\pmod{4}$, its $\delta$-invariant is $0$.\hfill$\Box$
\end{remark}

\section{Appendix}\label{sec:app}

Here are two more explicit examples of pairs $(B,X)$ with $B,X\in M_n(\Z)$ such that $X$ is unimodular, $B$ is skew-symmetric and $B$ is not mutation equivalent to $XBX^t$, distinguished by the $\delta$-invariant.

\noindent For $n=9$:

$B=\left(\begin{array}{rrrrrrrrr}
0 & 4 & 13 & -3 & -22 & 35 & -74 & 201 & 190 \\
-4 & 0 & 0 & 2 & -31 & -61 & 54 & -48 & 32 \\
-13 & 0 & 0 & -1 & -1 & 25 & -37 & 77 & 55 \\
3 & -2 & 1 & 0 & 12 & 93 & -123 & 230 & 124 \\
22 & 31 & 1 & -12 & 0 & 109 & -129 & 211 & 71 \\
-35 & 61 & -25 & -93 & -109 & 0 & -92 & 118 & 18 \\
74 & -54 & 37 & 123 & 129 & 92 & 0 & -514 & -175 \\
-201 & 48 & -77 & -230 & -211 & -118 & 514 & 0 & 97 \\
-190 & -32 & -55 & -124 & -71 & -18 & 175 & -97 & 0
\end{array}\right)$

$X=\left(\begin{array}{rrrrrrrrr}
4 & -14 & -38 & -20 & 73 & 76 & 204 & 807 & 2977 \\
3 & -5 & -22 & -5 & 18 & 43 & 119 & 546 & 2214 \\
-5 & 14 & 49 & 14 & -47 & -102 & -296 & -1316 & -5267 \\
5 & -5 & -25 & -4 & 17 & 46 & 113 & 555 & 2307 \\
2 & -3 & -16 & -6 & 21 & 27 & 73 & 306 & 1170 \\
-2 & 5 & 17 & 5 & -18 & -33 & -88 & -398 & -1590 \\
-1 & 4 & 10 & 6 & -22 & -20 & -53 & -201 & -718 \\
-4 & 13 & 45 & 14 & -50 & -88 & -246 & -1084 & -4285 \\
-3 & 8 & 28 & 8 & -28 & -54 & -151 & -674 & -2684
\end{array}\right).$

\noindent For $n=13$, we can choose the skew-symmetric matrix $B$ to be
$$\left(\begin{array}{rrrrrrrrrrrrr}
0 & 2 & -2 & 11 & -31 & -35 & 128 & -408 & 86 & -316 & -1288 & -4120 & -8869 \\
-2 & 0 & 1 & -12 & 42 & 65 & -255 & 706 & 17 & -168 & 2475 & 11745 & 26626 \\
2 & -1 & 0 & 10 & -33 & -49 & 190 & -534 & 3 & 59 & -1851 & -8455 & -19089 \\
-11 & 12 & -10 & 0 & -6 & 6 & -45 & 53 & 121 & -541 & 368 & 4295 & 10366 \\
31 & -42 & 33 & 6 & 0 & 18 & -136 & 157 & 385 & -1714 & 1116 & 13387 & 32336 \\
35 & -65 & 49 & -6 & -18 & 0 & -32 & 293 & -352 & 1512 & 502 & -5255 & -13801 \\
-128 & 255 & -190 & 45 & 136 & 32 & 0 & -327 & 34 & -98 & -1084 & -4297 & -9552 \\
408 & -706 & 534 & -53 & -157 & -293 & 327 & 0 & -211 & 1041 & -3078 & -18715 & -43439 \\
-86 & -17 & -3 & -121 & -385 & 352 & -34 & 211 & 0 & 332 & -539 & -4019 & -9493 \\
316 & 168 & -59 & 541 & 1714 & -1512 & 98 & -1041 & -332 & 0 & -1229 & -15981 & -38721 \\
1288 & -2475 & 1851 & -368 & -1116 & -502 & 1084 & 3078 & 539 & 1229 & 0 & 11612 & 28373 \\
4120 & -11745 & 8455 & -4295 & -13387 & 5255 & 4297 & 18715 & 4019 & 15981 & -11612 & 0 & 23684 \\
8869 & -26626 & 19089 & -10366 & -32336 & 13801 & 9552 & 43439 & 9493 & 38721 & -28373 & -23684 & 0
\end{array}\right)$$
and the unimodular matrix $X$ to be
$$X=\left(\begin{array}{rrrrrrrrrrrrr}
-5 & 9 & 10 & -64 & 112 & -539 & 1968 & 4901 & 7622 & 10674 & 28116 & 7513 & 25176 \\
2 & 1 & 13 & -80 & 68 & -436 & 1745 & 4451 & 7043 & 10350 & 26013 & 8751 & 31233 \\
3 & -3 & -2 & 10 & -28 & 126 & -446 & -1102 & -1687 & -2315 & -6263 & -1421 & -4527 \\
-3 & 0 & -7 & 51 & -32 & 237 & -982 & -2530 & -4041 & -6051 & -14906 & -5465 & -19870 \\
2 & -5 & -12 & 73 & -86 & 486 & -1873 & -4735 & -7425 & -10714 & -27461 & -8411 & -29361 \\
2 & -3 & -2 & 13 & -31 & 137 & -483 & -1191 & -1838 & -2519 & -6778 & -1605 & -5161 \\
-1 & 2 & 11 & -60 & 64 & -380 & 1486 & 3768 & 5915 & 8584 & 21895 & 6870 & 24140 \\
0 & 2 & 11 & -69 & 68 & -405 & 1597 & 4057 & 6399 & 9331 & 23634 & 7676 & 27175 \\
-1 & 3 & 11 & -66 & 75 & -423 & 1629 & 4110 & 6449 & 9287 & 23828 & 7287 & 25411 \\
2 & -3 & -8 & 48 & -60 & 332 & -1275 & -3225 & -5053 & -7284 & -18687 & -5691 & -19838 \\
-1 & -4 & -19 & 118 & -110 & 681 & -2700 & -6874 & -10846 & -15869 & -40096 & -13149 & -46672 \\
0 & 2 & 13 & -77 & 79 & -469 & 1836 & 4653 & 7319 & 10629 & 27053 & 8590 & 30254 \\
3 & -4 & -5 & 28 & -52 & 250 & -907 & -2260 & -3504 & -4901 & -12943 & -3391 & -11301
\end{array}\right).$$

Examples for higher $n$ can be readily found, but they do not fit in this document with the standard fontsize. The examples above hopefully illustrate that some pairs of congruent matrices can be shown to be not mutation equivalent with the $\delta$-invariant.

The following function ``Non\_MutationEquivalent\_but\_congruent'', written in Python, can generate examples of such $B,X$ with SageMath 9.0:

\begin{lstlisting}[language=Python]
def triu(m):
    t = matrix(m.base_ring(), m.nrows(), sparse=True)
    for (i,j) in m.nonzero_positions():
        if i < j:
            t[i,j] = m[i,j]
    return t+matrix.identity(m.nrows())

def Non_MutationEquivalent_but_congruent(iter,size):
    for i in range(0,iter):
        N=random_matrix(ZZ, size, algorithm='unimodular', upper_bound=50000);
        V=triu(N);
        B=V-V.transpose();
        C=V+V.transpose();
        print('The matrix B is:');
        show(B);
        d=C.determinant();
        X=random_matrix(ZZ, size, algorithm='unimodular', upper_bound=50000);
        print('The matrix X is:');
        show(X);
        Bp=X*B*(X.transpose());
        Vp=triu(Bp);
        Cp=Vp+Vp.transpose();
        dp=Cp.determinant();
        print('The mod 4 invariant of B is:',mod(d,4));
        print('The mod 4 invariant of XBXt is:',mod(dp,4));
        if mod(dp,4)!=mod(d,4):
            print('The invariants are different for this congruent pair.');
        if mod(dp,4)==mod(d,4):
            print('The invariants are equal for this congruent pair.');
        print('------------ End of case ------------:', i);
\end{lstlisting}

\bibliographystyle{alpha}
\bibliography{main}

\end{document}